\documentclass{amsart}
\usepackage[utf8]{inputenc}
\usepackage{amsmath,amssymb,amsbsy,amsfonts,latexsym,amsopn,amstext,cite,
	amsxtra,euscript,amscd,bm}
\usepackage{amsmath,amssymb,amsbsy,amsfonts,amsthm,latexsym,
	amsopn,amstext,amsxtra,euscript,amscd,mathrsfs}
\usepackage{mathtools}
\usepackage[colorinlistoftodos,prependcaption,textsize=tiny]{todonotes}
\usepackage{url}
\usepackage{empheq}
\usepackage{bbm}
\usepackage{tikz}
\usetikzlibrary{scopes,arrows.meta,arrows,decorations.pathmorphing,backgrounds,positioning,fit,petri,shapes,calc}
\usepackage[colorlinks,linkcolor=blue,anchorcolor=blue]{hyperref}

\newcommand{\tend}[3][]{\xrightarrow[#2\to#3]{#1}}

\newcommand{\setdef}{\stackrel {\rm {def}}{=}}
\newcommand{\ds}{\displaystyle}

\newcommand{\N}{\mathbb N}
\newcommand{\Z}{\mathbb Z}
\newcommand{\C}{\mathbb C}
\newcommand{\Q}{\mathbb Q}
\newcommand{\R}{\mathbb R}
\newcommand{\E}{\mathbb E}
\newcommand{\T}{\mathbb T}

\newcommand{\Rea}{\textrm{Re}}
\newcommand{\dprime}{\prime\prime}

\title{Mixing sequences for non-mixing transformations and group actions}

\author{\MakeLowercase{el} Houcein \MakeLowercase{el} Abdalaoui}
\address{University of Rouen Normandy
	Department of Mathematics, LMRS  UMR 60 85 CNRS\\
	Avenue de l'Universit\'e, BP.12
	76801 Saint Etienne du Rouvray - France .}
\email{elhoucein.elabdalaoui@univ-rouen.fr}
\urladdr{http://www.univ-rouen.fr/LMRS/Persopage/Elabdalaoui/}
\author{Terry Adams}
\address{Department of Mathematics and Statistics\\
	State University of New York
	Albany, NY 12222 }
\email{terry@ieee.org}
\date{May 2021}
\dedicatory{\textbf{Dedicated to the 80th Anniversary of Professor Jean-Paul Thouvenot}}

\newtheorem{thm}{Theorem}[section]
\newtheorem{prop}{Proposition}[section]
\newtheorem{lem}{Lemma}[section]

\newtheorem{cor}{Corollary}[section]
\newtheorem{question}{Question}[section]
\newtheorem{defn}{Definitions}[section]
\newtheorem{rem}{Remark}[section]
\begin{document}
	\begin{abstract} 
		We establish that there are non-mixing maps that are mixing on appropriate sequences including sequences $(s_i)$ which satisfy the Rajchman dissociated property. Our examples are based on the staircase rank one construction, $M$-towers constructions and the Gaussian transformations.  As a consequence, we obtain there are non-mixing maps which are mixing along the squares. We further prove that a sequence $M=(m_n)$ is a mixing sequence for some weak mixing ${1}/{2}$-rigid transformation $T$ if and only if 
		the complement of $M$ is a thick set. This result is generalized to ${r}/{(r+1)}$-rigid transformations for $r\in \N$.  Moreover, by applying Host-Parreau characterization of the set of continuity from Harmonic Analysis, we extend our results to the infinite countable abelian group actions.      
	\end{abstract}


\maketitle

\section{Introduction}

The purpose of this paper is to bring a negative answer to the following question asked by V. Bergleson and al. \cite{B}.
\begin{question}
	Does mixing along the squares imply mixing?
\end{question} 
This question seems to be related to the mixing of height order for countable and  amenable group action. In fact, we give a complete characterization for a sequence $s_i$ of natural numbers to be a mixing sequence for some partially rigid, weak mixing transformation.  Note, if $s_i$ contains bounded gaps, and a probability preserving transformation $T$ is mixing on $s_i$, then $T$ must be strongly mixing.  We show that if the sequence $s_i$ does not contain bounded gaps, then there exists a non-mixing $T$ which is mixing on $s_i$.  We will further establish that the subset $M$ of integers for which there is a non-mixing map which is mixing on it should have the Rajchman dissociated property, that is , its complement contain a shift of a 
some set $\Big\{\ds\sum \epsilon_jn_j,  \epsilon_j\in \{\pm 1,0\},  \ds \sum |\epsilon_j|<\infty\Big\}$, where $(n_j)$ is a dissociated sequence (a sequence $(n_j)$ is a dissociated sequence if  $\ds n_j >2 \sum_{i=1}^{j-1}n_i,$ for each $j \geq 1$).  Let us notice that this class contain a subsets $M$ of density zero. Moreover, it is well known  that if $S$ is a Sidon set and $T$ is mixing on its complement then it $T$ must be mixing. This later result and much more is well known in harmonic analysis. In fact, Host-Parreau proved that if a set $M$ does not have the Rajchman dissociated property then if $T$ is mixing on its complement then $T$ should be mixing (in the Harmonic Analysis language, this result is stated as a characterization of the so-called sets of continuity and sets of Rajchman.). 
For more details we refer to section \ref{Gaussian}. We further present an extension of our results to the case of countable abelian group actions in section \ref{Groupe}.  Our contributions improve and generalize that of V. V. Ryzhikov \cite{R-rigid-m} to Bergelson's question. 

\section{Mixing sequences and ${1}/{2}$-rigid transformations}\label{vs-staircase-alpha}
In this section, we prove the following. 

\begin{thm}
	\label{v-staircase-thm}
	Given an increasing sequence of natural numbers 
	$M = \{ m_n \in \N: n\in \N \}$, 
	there exists a weak mixing ${1}/{2}$-rigid 
	transformation $T$ that is mixing on $M$ if and only if 
	the complement of $M$ is a thick set. 
\end{thm}
\noindent{}We recall
\begin{defn}
	A set $K \subset \N$ is thick if for each $\ell \in \N$, there exists 
	$k\in \N$ such that 
	\[
	\{ k + i : -\ell \leq i \leq \ell \} \subset K . 
	\]
\end{defn}
Note, a set $K$ is thick if and only if its complement is not syndetic.  
\begin{cor}
There exist non-mixing transformations that are mixing along the squares 
$n^2$, $n \in \N$. 
\end{cor}


\begin{cor}
Let $K$ be a thick subset of the natural numbers.  
There exists a $\frac{1}{2}$-rigid weak mixing transformation 
$T$ that is mixing on the complement of $K$.  
\end{cor}

The partial rigidity constant $\frac{1}{2}$ is the maximum 
rigidity constant for which the previous corollary holds.  
In particular, it is not difficult to prove the following 
proposition.  We provide its proof for the convenience of the reader.

\begin{prop}\label{half}
There exists a thick subset $K$ of the natural numbers 
such that if $T$ is any invertible measure preserving 
$\alpha$-rigid transformation with $\alpha > {1}/{2}$, 
then $T$ is not mixing on the complement of $K$. 
\end{prop}
\begin{proof}
Let $K$ be any thick subset with the property that 
if $k \in K$, then $2k \notin K$.  For example, 
$K$ may be the set, 
\begin{align}\label{VV}
K = \bigcup_{n=1}^{\infty} [2^{2n}, 2^{2n+1}-1] \cap \N . 
\end{align}
Suppose invertible measure preserving transformation $T$ is 
$\alpha$-rigid with $\alpha > {1}/{2}$.  
Thus, there exists a sequence $r_n$ of positive integers such that 
for all measurable sets $A$, 
\[
\lim_{n\to \infty} \mu (T^{r_n}A\cap A) \geq \alpha \mu (A). 
\]
If $r_n \notin K$ for infinitely many $n$, then $T$ is not mixing 
on $K^c$ and we are done.  

Otherwise, $r_n \in K$ for sufficiently large $n$.  
Let $\beta$ be such that ${1}/{2} < \beta < \alpha$ and let 
$A$ be a set of positive measure such that $\mu (A) < 2\beta - 1$.  
Since $T$ is $\alpha$-rigid, there exists $N \in \N$ such that 
for $n \geq N$, $r_n \in K$ and 
\[
\mu (T^{r_n}A\cap A) > \beta \mu (A) . 
\]
Hence, for $n \geq N$, 
$\mu (A\cap T^{-r_n}A) > \beta \mu (A)$ and 
$\mu (A\cap T^{r_n}A) > \beta \mu (A)$.  
This implies 
\begin{align}
\mu (T^{2r_n}A\cap A) &\geq 
\mu (A\cap T^{-r_n}A \cap T^{r_n}A) \\ 
&> 2\big( \beta - \frac{1}{2}\big) \mu(A) \\ 
&= \big( 2\beta - 1\big) \mu (A) > \mu (A)^2 . 
\end{align}
Therefore, $T$ is not mixing on $2r_n \in K^c$ for $n\geq N$. 
\end{proof} 
An alternative proof of Proposition \ref{half} can be obtained by applying Theorem 5 from \cite{elA}. Therein, it is proved that any $\alpha$-rigid map with $\alpha>1/2$ is spectrally disjoint from any mixing map. 
Let us mention that Theorem 1 in \cite{R-rigid-m} required that the complement of the set $M$ for which one can construct a rigid map contain
$$\bigcup_{i=0}^{+\infty}[a_i,a_i+L_i] \cup[a_i,a_i+L_i]\cup \cdots [ia_i,ia_i+L_i].$$ 
for some $a_i, L_i \longrightarrow +\infty$ as $i\longrightarrow +\infty$. Obviously, such set is a thick set, but it is easy to see that the set 
$K$ in \eqref{VV} is thick and it is not satisfy this property. \\

One direction of Theorem \ref{v-staircase-thm} is straightforward 
and follows from the following proposition. 
\begin{prop}
\label{prop-2}
If a finite measure preserving transformation $T$ is mixing 
on a syndetic sequence $m_n$, $n\in \N$, then $T$ is strong mixing. 
\end{prop}
{\bf Proof:} 
Suppose $T$ is mixing on an increasing sequence $M=\{m_1, m_2, \ldots \}$ 
and $M$ has bounded gaps.  There exists $d$ such that $m_{n+1}-m_{n} \leq d$ 
for all $n \in \N$.  This implies for any measurable sets $A$ and $B$ 
and fixed $j$ that
\[
\lim_{n\to \infty} \mu ( T^{m_n + j}A\cap B ) = \lim_{n\to \infty} \mu ( T^{m_n}A\cap T^{-j}B ) = \mu ( A ) \mu ( T^{-j} B ) = \mu ( A ) \mu ( B ) .
\]
This proves that $m_n + j$ is a mixing sequence for each $0\leq j \leq d$.  Therefore, $\bigcup_{j=0}^{d} (M+j) = \N$ is a mixing sequence and $T$ is strong mixing. $\Box$ 

If the complement of $M$ is not thick, then $M$ is syndetic and every 
finite measure preserving transformation that is mixing on $M$, must be strong mixing. 
This proves one direction in Theorem \ref{v-staircase-thm}. 

\subsection{Staircase constructions}
\label{stair-construct-section}
Our ${1}/{2}$-rigid weak mixing transformations are constructed 
as a modification of the staircase transformations. 
Staircase transformations are rank-one transformations 
constructed via cutting and stacking.  Given a sequence 
$r_n$ of cut parameters, at the $n^{th}$ stage, cut the 
tower of intervals into $r_n$ subcolumns, add $j$ spacers 
on the $j^{th}$ subcolumn for $0\leq j \leq r_n - 1$ and 
stack left to right.  Note, we are adding zero spacers 
on the first subcolumn. 
If $r_n$ does not grow too rapidly, this defines an ergodic 
invertible measure preserving transformation on a standard 
probability space (i.e., $[0,1]$). 

\subsection{${1}/{2}$-rigid staircase constructions}
If the cut parameter $r_n = 2$ for infinitely many $n$, 
then the staircase transformation is ${1}/{2}$-rigid, 
and hence, not strong mixing.  
To obtain the counterexamples, we will set $r_n=2$ 
for rare $n\in \N$. 

Before proceeding with a proof of Theorem \ref{v-staircase-thm}, 
we will define three parameterized families 
of staircase transformations.  
We give the families the following identifiers: 
$\mathcal{S}$, $\mathcal{R}$ and $\mathcal{T}$.  
The final counterexamples are derived from the family $\mathcal{T}$, 
and the other two families are instrumental for obtaining 
the counterexamples. 

\subsubsection{Finite staircase transformations}
First, we define a family of transformations called $\mathcal{R}$.  
Given a positive integer $\alpha \geq 2$ and a column $C$ 
of intervals of equal width, a transformation $R_{\alpha,C}$ is defined 
using cutting and stacking in the following manner.  
First, cut column $C_0=C$ into $\alpha$ subcolumns of equal width, 
add a staircase of height $\alpha-1$ on top of the subcolumns 
and stack from left to right.  This produces column $C_1$.  
Column $C_{n+1}$ is produced in a similar manner from column $C_n$ 
by cutting $C_n$ into $\alpha$ subcolumns of equal width, 
adding a staircase of height $\alpha-1$ and stacking from left to right.  
This produces a rank-one finite measure preserving transformation, 
normalized to be defined a.e. on $[0,1]$.  
This transformation will be partially rigid, hence not strongly 
mixing. 
It will also have the property of being uniform Cesaro 
which we define here\footnote{This notion was introduced by the second author and N. Friedman here \cite{AF}.}.  

\begin{defn}
A transformation $R$ is uniform Cesaro if for all measurable sets $A$, 
\[
\lim_{N\to \infty} \sup_{q\in \N} \int_X 
\Big| \frac{1}{N} \sum_{i=0}^{N-1} \mathbbm{1}_A\big( R^{iq}x \big) - \mu(A) \Big| d\mu 
= 0 . 
\]
\end{defn}

\subsubsection{Mixing staircase transformations}
Now we define a subfamily of staircase transformations 
denoted by $\mathcal{S}$.  These transformations 
will be defined using 3 sequences of parameters, 
denoted $\vec{\alpha} = (\alpha_p)_{p=1}^{\infty}$, 
$\vec{\beta}=(\beta_p)_{p=1}^{\infty}$ and a binary sequence 
$\vec{b}=(b_p)_{p=1}^{\infty}$.  Also, $\beta_1=1$.  
Given these sequences, 
a transformation $S=S_{\alpha,\beta,b} \in \mathcal{S}$ 
is defined where $r_n = \alpha_p + b_p$ for $\beta_p \leq n < \beta_{p+1}$.  
Each transformation $S\in \mathcal{S}$ will be produced 
recursively by alternating between two different types of steps.  
One step is to fix the cut parameter $\alpha$ and use the 
fact that $R_{\alpha,C}$ is uniform Cesaro to choose 
the next cut parameter.  In the second step, the cut parameter 
sequence is modified by a binary vector to approximate 
an element from $K \subset M^c$ (define by eq. \eqref{T-def} in subsection \ref{Diagram}).
These steps are repeated infinitely often to produce 
a transformation $S$.  We guarantee $S$ is mixing, 
since it will not be difficult to enforce the following 
conditions, $\lim_{n\to \infty} r_n = \infty$ and 
\[
\lim_{n\to \infty} \frac{r_n^2}{h_n} = 0 . 
\]
By \cite{Adams}, the transformation $S$ will be strongly mixing. 

\subsection{Supporting lemma}
Prior to proving Theorem \ref{v-staircase-thm}, we will define 
some notation and prove a couple lemmas on the combinatorics 
of the cut parameters and corresponding heights. 

To define a staircase transformation, an infinite sequence 
$r_n\geq 2$, $n\in \N$, is specified.  We may denote 
this sequence using vector notation as 
\[
\vec{r} = (r_1, r_2, \ldots ) . 
\]
For $i\in \N$, let $\vec{r}(i) = r_i$.  
Also, for $i\in \N$, let $\vec{e_i}$ be 
the vector with a one in the $i^{th}$ position and zeros 
everywhere else.  
Each vector of cut parameters and initial value $h_1$, 
determines a vector of heights given by the formula 
\[
h_{n+1} = r_n h_n + \frac{r_n(r_n-1)}{2} . 
\]
We may refer to this vector as 
$\vec{h} = (h_1, h_2, \ldots )$ and use the function 
$\mathcal{H}$ to define the mapping 
\[
\vec{h} = \mathcal{H} ( \vec{r}, h_1 ) . 
\]
For convenience, let 
$\mathcal{H}_i (\vec{r},h_1 ) = \vec{h}(i)$. 
We define a nondecreasing vector $\vec{r}$ as a vector 
whose components are nondecreasing 
(i.e., $r_1 \leq r_2 \leq \ldots$).  

\begin{lem}
\label{v-s-lem1}
Suppose $\vec{r}$ is a nondecreasing vector 
and $h_1\in \N$.  
If $j, k,n \in \N$ such that 
$\mathcal{H}_{n}(\vec{r},h_1) \leq k < \mathcal{H}_{n+1}(\vec{r},h_1)$ and 
\begin{align}
k <& \Big( \Pi_{i=j}^{n-1} \big( r_i + 1\big) \Big) h_j ,\label{lem-eq1}
\end{align}
then there exists a binary vector 
\begin{align}
\vec{b} =& \sum_{i=j}^{n-1} b_i \vec{e}_i 
\end{align}
such that 
\begin{align}
0 \leq k - \mathcal{H}_{n}(\vec{r}+\vec{b},h_1) <& 
\Big( \frac{1}{r_j} + \frac{1}{\mathcal{H}_j(\vec{r},h_1)} \Big) 
\mathcal{H}_{n}(\vec{r}+\vec{b},h_1) \label{lem-eq2}. 
\end{align}
\end{lem}
This lemma says that if increasing all of the cut parameters 
in a range by one surpasses the value $k$, then it is possible 
to modify certain cut parameters by one, so that 
a new height parameter approximates $k$.
\begin{proof}
Let $\vec{h} = \mathcal{H} (\vec{r},h_1)$.  
Suppose we modify only the $i^{th}$ component of $\vec{r}$.  
Thus, 
\begin{align*}
\mathcal{H}_{i+1}(\vec{r}+\vec{e_i},h_1) - \mathcal{H}_{i+1}(\vec{r},h_1) 
=& \big( r_i+1 \big) h_i + \frac{(r_i+1)r_i}{2} - 
\Big( r_i h_i + \frac{r_i(r_i-1)}{2} \Big) \\ 
=& h_i + r_i \\ 
<& \Big( \frac{1}{r_i} + \frac{1}{h_i} \Big) h_{i+1} 
\end{align*}
Notice that we can still cap the difference, as we compute 
forward values of $\mathcal{H}$.  
For the next equation,
let $h^{\dprime}=\mathcal{H}(\vec{r}+\vec{e_i}+\vec{e_{i+1}},h_1)$ 
and $h^{\prime}=\mathcal{H}(\vec{r}+\vec{e_{i+1}},h_1)$. 
Then 
\begin{align*}
\mathcal{H}_{i+2}(\vec{r}+&\vec{e_i}+\vec{e_{i+1}},h_1) - 
\mathcal{H}_{i+2}(\vec{r}+\vec{e_{i+1}},h_1) \\ 
=& \big( r_{i+1}+1 \big) h_{i+1}^{\dprime} + \frac{(r_{i+1}+1)r_{i+1}}{2} - 
\Big( (r_{i+1}+1) h_i^{\prime} + \frac{(r_{i+1}+1)r_{i+1}}{2} \Big) \\ 
=& \ \big( r_{i+1}+1 \big) \big( h_{i+1}^{\dprime} - h_{i+1}^{\prime} \big) \\ 
\leq& \big( r_{i+1}+1\big) \Big( \frac{1}{r_i} + \frac{1}{h_i} \Big) 
h_{i+1} \\ 
=& \Big( \frac{1}{r_i} + \frac{1}{h_i} \Big) \big( r_{i+1}+1 \big) h_{i+1} \\ 
<& \Big( \frac{1}{r_i} + \frac{1}{h_i} \Big) h_{i+2}^{\prime}
\end{align*}
Let $B = \{ \vec{b}=\sum_{i=j}^{n-1} b_i \vec{e}_i : b_i \in \{0,1\} \}$.  
Let $\vec{b}_0$ be such that for each $\vec{b}\in B$, 
\[
\mathcal{H}_{n}(\vec{r}+\vec{b}, h_1) \leq \mathcal{H}_{n}(\vec{r}+\vec{b_0}, h_1) 
\leq k . 
\]
The binary vector $b_0$ is chosen to maximize the modified height 
without going over $k$.  
We prove by contradiction that $\vec{b}=\vec{b}_0$ satisfies 
the condition (\ref{lem-eq2}) of the lemma.  
Suppose it does not.  We know that all $b_i$ are not 1, since condition 
(\ref{lem-eq1}) would imply the modified height would then be greater than $k$.  
Let $i_0$ be the maximum $i_0\leq n$ such that $b_{i_0}=0$.  
Then we can set $b_{i_0}=1$, and the new height would still be less 
than $k$. 
This is true, since by the same argument above, the following holds 
\[
\mathcal{H}_{n}(\vec{r}+\vec{b_0}+\vec{e_{i_0}},h_1) - 
\mathcal{H}_{n}(\vec{r}+\vec{b_0},h_1) < 
\Big( \frac{1}{r_{i_0}} + \frac{1}{h_{i_0}} \Big) 
\mathcal{H}_{n}(\vec{r}+\vec{b_0},h_1) . 
\]
Hence, 
\begin{align*}
&\mathcal{H}_{n}(\vec{r}+\vec{b_0}+\vec{e_{i_0}},h_1) < 
\mathcal{H}_{n}(\vec{r}+\vec{b_0},h_1) + 
\Big( \frac{1}{r_{i_0}} + \frac{1}{h_{i_0}} \Big) 
\mathcal{H}_{n}(\vec{r}+\vec{b_0},h_1) \\ 
\leq& k - \Big( \frac{1}{r_{j}} + \frac{1}{h_{j}} \Big) 
\mathcal{H}_{n}(\vec{r}+\vec{b_0},h_1) + 
\Big( \frac{1}{r_{i_0}} + \frac{1}{h_{i_0}} \Big) 
\mathcal{H}_{n}(\vec{r}+\vec{b_0},h_1) \leq k . 
\end{align*}
This proves the lemma by contradiction. 
\end{proof}

\subsection{Flow Diagram}\label{Diagram}
In this section, we give an overview of the construction 
of ${1}/{2}$-rigid transformations.  The construction uses 
a nested infinite recursion.  For the inner recursion, 
we alternate between a uniform Cesaro substage and a number 
approximation stage.  The number approximation stage is used 
to closely approximate infinitely many values from the thick 
subset $K$ using column heights of this transformation.  
These two alternating steps are repeated infinitely 
often to produce a mixing staircase transformation.  
Then we use the property of mixing to carefully select 
a height to introduce a ${1}/{2}$-rigid time (i.e., cut in half 
and stack without adding spacer).  

The mixing transformation is interrupted, and we return to the 
previous stage with two substages (uniform Cesaro $\to$ number approximation). 
Figure \ref{fig:flow-diagram} is a diagram showing the nested recursion loops. 
It is straightforward that the resulting transformation is ${1}/{2}$ rigid. 
Since the complement of $M$ is a thick subset, there exist
increasing sequences $k_n$ and $\ell_n$ of natural numbers such that 
\[
k_n + i \notin M\ \ \mbox{for}\ -\ell_n \leq i \leq \ell_n . 
\] 
Set 
\begin{align}\label{T-def}
K = \{ k_n: n\in \N \} \subset M^c . 
\end{align}

In the construction, we will use a sequence $\epsilon_i > 0$ 
of real numbers such that 
\[
\sum_{i=1}^{\infty} \epsilon_i < \infty . 
\]







\begin{figure}
\begin{tikzpicture}

\node (rect) at (2.2,1)[draw, thick, minimum width=6cm, minimum height=3cm]{};

\node at (.6,1)[circle, draw, minimum size=1cm, aspect=.5, inner sep=0pt, outer sep=0pt, text width=1.8cm] {
\begin{minipage}{3.0cm}
\parbox{1.8cm}{\begin{center}{\small Uniform\\Cesaro\\substage}\end{center}}
\end{minipage}
};
\node at (3.7,1)[circle, draw, minimum size=1cm, aspect=1, inner sep=1pt, outer sep=-1pt, text width=1.8cm] {
\begin{minipage}{3.0cm}
\parbox{1.8cm}{\begin{center}{\small Number\\ approximation substage}\end{center}}
\end{minipage}
};
\draw[<-] (1.3,2.1) .. controls +(up:0.5cm) and +(right:.1cm) .. (3.0,2.2) node[above](){};
\draw[->] (1.2,-0.2) .. controls +(down:0.5cm) and +(left:.4cm) .. (2.8,-0.1) node[above](){};

\node (rect) at (9.7,1)[draw, thick, minimum width=3cm, minimum height=3cm]{
\begin{minipage}{2.4cm}
{\small Selection of $\frac{1}{2}$-rigid time}
\end{minipage}
};

\draw[->] (5.5,1) -- (8.0,1) node[above](){\parbox{5.0cm}{{\small Mixing\\ Transformation}}};
\draw[-{Stealth[length=3mm, width=2mm]}] (9.7,-0.7) .. controls +(down:1cm) and +(down:1cm) .. (2.2,-0.7) node[below,pos=.46](){
\parbox{8cm}{{\small Interrupt mixing construction and repeat process}}};

\end{tikzpicture}
\caption{Construction Flow Diagram}
\label{fig:flow-diagram}
\end{figure}
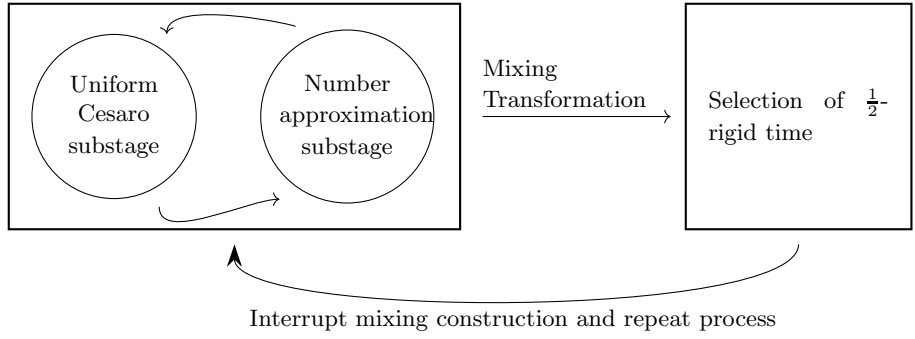

\subsection{Uniform Cesaro and mixing}
We describe two substages that are used recursively to construct 
mixing transformations.  The procedure for constructing the mixing 
transformations is then modified in an infinite loop to produce 
a ${1}/{2}$-rigid transformation. 

\subsubsection{$p$-stage uniform Cesaro}
\label{UC-substage}
Let $C_p$ be a column of intervals of height $H_p$ and $P_p$ be the partition 
consisting of $H_p + 1$ elements, $H_p$ elements for each interval of $C_p$, 
and the complement of $C_p$.  Given $\alpha_p \in \N$, let 
$R = R_{\alpha_p,C_p}$ be the uniformly Cesaro transformation with 
parameters $\alpha_p$ and $C_p$.  There exists $N_p \in \N$ such that 
for $n \geq N_p$ and $A\in P_p$, 
\begin{align}
\sup_{q\in \N} \int_X \Big| \frac{1}{n} \sum_{i=0}^{n-1} 
\mathbbm{1}_A \big( R^{-iq} x \big) - \mu(A) \Big| d\mu 
\leq \epsilon_{H_p} \mu(A) . \label{UC-cond}
\end{align}

\subsubsection{Number approximation stage}
\label{NA-substage} 
Let $\rho_p$ be a natural number such that 
\[
\frac{1}{\rho_p} < \epsilon_p . 
\]
Choose $j_p \in \N$ such that the height 
$h_{j_p}$ produced from the transformation $R_{\alpha_p,C_p}$ 
from the previous substage (\ref{UC-substage}) satisfies, 
\[
\frac{(\rho_p + 1)^2}{h_{j_p}} < \epsilon_p . 
\]
At this point, set $r_n = \rho_p$ for $n\geq j_p$.  
Thus, 
\[
h_{j_p+1} = \rho_p h_{j_p} + \frac{\rho_p(\rho_p-1)}{2} 
\]
and for $n \geq j_p$, 
\[
h_{n} = \rho_p^{n-j_p} h_{j_p} + \frac{\rho_p(\rho_p^{n-j_p}-1)}{2} . 
\]
This implies 
\begin{align}
\lim_{n\to \infty} \frac{(\rho_p+1)^{n-j_p} h_{j_p}}{ h_{n+1} } &= 
\lim_{n\to \infty} \frac{(\rho_p+1)^{n-j_p} h_{j_p}}{ \rho_p^{n-j_p+1}h_{j_p} + \frac{\rho_p(\rho_p^{n-j_p+1}-1)}{2} } \\
&\geq \lim_{n\to \infty} \frac{(\rho_p+1)^{n-j_p} h_{j_p}}{ \rho_p^{n-j_p+1}h_{j_p} + \rho_p^{n-j+2} } \\ 
&= \infty . 
\end{align}

Hence, for $k \in K$ sufficiently large, 
there exists $n=n(p,k)$ such that 
\[
h_n \leq k < h_{n+1} < \big( \rho_p + 1\big)^{n-j_p} h_{j_p} . 
\]
Choose one such $k\in K$ and place it in the set denoted $K^{\prime}$.  
By Lemma \ref{v-s-lem1}, there exists 
$\vec{b} = \sum_{i=j_p}^{n-1} b_i \vec{e}_i$ such that 
\begin{align}
0 \leq k - \mathcal{H}_{n}(\vec{r}+\vec{b},h_1) <& 
\Big( \frac{1}{r_{j_p}} + \frac{1}{\mathcal{H}_{j_p}(\vec{r},h_1)} \Big) 
\mathcal{H}_{n}(\vec{r}+\vec{b},h_1) . 
\end{align}
Set $\alpha_{p+1} = \rho_p$ and let $C_{p+1}$ be the resulting column 
of height $\mathcal{H}_n(\vec{r}+\vec{b}, h_1)$.  
Now, return to step (\ref{UC-substage}).  

This infinite loop produces an infinite set $K^{\prime} \subset K$.  

\subsection{Mixing transformation recursion}
\label{mix-trans-rec-section}
If the substages (\ref{UC-substage}) and (\ref{NA-substage}) are repeated 
infinitely often, then the resulting transformation is strongly mixing.  
This is due to the fact that $\lim_{n\to \infty} r_n = \infty$ and 
\[
\lim_{n\to \infty} \frac{r_n^2}{h_n} = 0 . \label{spacer-cond}
\]

Let $S_1 \in \mathcal{S}$ be the initial transformation produced from 
repeating substages (\ref{UC-substage}) and (\ref{NA-substage}), 
infinitely often.  Since $S_1$ is strong mixing, there 
exists a positive integer $L_1$ such that for $A,B \in P_1$ and $n\geq L_1$, 
\[
| \mu(T^n A\cap B) -\mu(A)\mu(B) | \leq \epsilon_1 \mu(A)\mu(B) . 
\]
Choose $y_1$ such that $\ell_{y_1} > L_1$ and $k_{y_1} \in K_1=K^{\prime}$ 
sufficiently large such that its corresponding $p=p_1$, $n=n(p,k_{y_1})$ 
and binary vector $\vec{b}_1$ satisfy 
\[
\epsilon_1 \mathcal{H}_{n} > L_1, 
\]
and 
\begin{align}
0 \leq k_{y_1} - \mathcal{H}_{n}(\vec{r}+\vec{b}_1,h_1) <& 
\Big( \frac{1}{r_{j_p}} + \frac{1}{\mathcal{H}_{j_p}(\vec{r},h_1)} \Big) 
\mathcal{H}_{n}(\vec{r}+\vec{b}_1,h_1) . 
\end{align}
At this point, we modify the mixing construction $S_1$.  
Consider its column, denoted $D_1$ of height 
$\mathcal{H}_{n}(\vec{r}+\vec{b}_1,h_1)$.  
Add $k-\mathcal{H}_{n}(\vec{r}+\vec{b}_1,h_1)$ spacers to the top of $D_1$ 
and let $\mathcal{H}_n$ be the height of this new column. 
Then cut $D_1$ in half and stack the right half on top of the left half.  
Return to the infinite recursion of alternating substages 
(\ref{UC-substage}) and (\ref{NA-substage}) with $\alpha = r_n-1$ 
to produce a mixing transformation denoted by $S_2$.  

For the general step, given a strong mixing transformation 
$S_i \in \mathcal{S}$, 
there exists a positive integer $L_i$ such that 
for $A,B \in P_i$ and $n\geq L_i$, 
\[
| \mu(T^n A\cap B) -\mu(A)\mu(B) | \leq \epsilon_i \mu(A)\mu(B) . 
\]
Choose $y_i$ such that $\ell_{y_i} > L_i$ and $k_{y_i} \in K_i=K^{\prime}$ 
sufficiently large such that there exists $p_i > p_{i-1}$, $n_i$ 
and binary vector $\vec{b}_i$ such that 
\[
\epsilon_i \mathcal{H}_{n_i} > L_i, 
\]
and 
\begin{align}
0 \leq k_{y_i} - \mathcal{H}_{n_i}(\vec{r}+\vec{b}_i,h_1) <& 
\Big( \frac{1}{r_{j_{p_i}}} + \frac{1}{\mathcal{H}_{j_{p_i}}(\vec{r},h_1)} \Big) 
\mathcal{H}_{n_i}(\vec{r}+\vec{b}_i,h_1) . \label{meas-add}
\end{align}
At this point, we modify the mixing construction $S_i$.  
Consider its column, denoted $D_i$ of height 
$\mathcal{H}_{n_i}(\vec{r}+\vec{b}_i,h_1)$.  
Add $k-\mathcal{H}_{n_i}(\vec{r}+\vec{b}_i,h_1)$ spacers to the top of $D_i$ 
and let $\mathcal{H}_{n_i}$ be the height of this new column, renamed $D_i$.  
Then cut $D_i$ in half and stack the right half on top of the left half.  
Return to the infinite recursion of alternating substages 
(\ref{UC-substage}) and (\ref{NA-substage}) with $\alpha = r_{n_i-1}$ 
to produce a mixing transformation denoted by $S_{i+1}$.  

The resulting transformation $T$ is finite measure preserving, since 
we do not add an infinite amount of measure.  In particular, 
the proportion of measure added on top of each column of 
height $\mathcal{H}_{n_i}$ satisfies
\begin{align}
\Big( \frac{1}{r_{j_{p_i}}} + \frac{1}{\mathcal{H}_{j_{p_i}}(\vec{r},h_1)} \Big) &\leq 
\frac{2}{r_{j_{p_i}}} \\ 
&= \frac{2}{\rho_{p_i}} < 2\epsilon_{p_i} . 
\end{align}
Since the sequence $p_i$ is strictly increasing, then the amount of measure added 
is summable.

\subsection{Proof of Theorem \ref{v-staircase-thm}}
One direction of the theorem follows from 
Proposition \ref{prop-2}.  For the other direction, 
assume $M \subset \N$ such that the complement of $M$ contains 
a thick subset.  Also, assume parameters have been generated 
as in the previous sections including 
$k_i, \ell_i, p_i, \vec{b}_i, n_i$ and $S_i$.  
Partition the natural numbers into two sets which we consider 
separately, 
\begin{align}
M_1 =& \bigcup_{i=1}^{\infty} \Big( [\epsilon_i \mathcal{H}_{n_i}, \mathcal{H}_{n_i}-L_i] \cup 
[\mathcal{H}_{n_i}+L_i, \mathcal{H}_{n_i+1}] \cap \N \Big) , \\ 
M_2 =& M \setminus M_1 . 
\end{align}

\subsubsection{$m_i \in M_2$}
For $m_i \in M_2$, there exists $h_j$ such that 
\[
h_j \leq m_i < h_{j+1} . 
\]
The columns $C_j$ and $C_{j+1}$ of heights $h_j$ and $h_{j+1}$ 
respectively both have staircases of heights 
$r_j$ and $r_{j+1}$ on top.  Due to the uniform Cesaro condition 
(\ref{UC-cond}) that we enforce, then 
for $A, B\in P_q$, 
\[
\lim_{m_i\in M_2, m_i \to \infty} 
\mu (T^{m_i} A\cap B) = \mu (A)\mu (B) . 
\]
See \cite{AF} for a detailed proof. 

\subsubsection{$m_i \in M_1$} 
This case may seem a bit more subtle.  We use the mixing condition 
applied to each $S_i$.  

First, separate this into two subcases:
\begin{itemize}
\item $m_i \in [\epsilon_i \mathcal{H}_{n_i}, \mathcal{H}_{n_i} - L_i]$, 
\item $m_i \in [\mathcal{H}_{n_i}+L_i, \mathcal{H}_{n_i+1}]$. 
\end{itemize}

We start with the first case. 
Let $D_{i,1}$ be the bottom $\mathcal{H}_{n_i}$ levels 
of column $D_i$.  Also, let $D_{i,2}$ be the left side of the top 
$m_i$ levels of $D_i$ and $D_{i,3}$ be the right side of the top 
$m_i$ levels of $D_i$.  
Let $A_{i,j} = A\cap D_{i,j}$ for $j=1,2,3$.  Thus, we have 
\[
\mu (T^{m_i} A_{i,1} \cap B) \leq \mu (S_i^{m_i} A_{i,1} \cap B) 
\leq \mu (S_i^{m_i} A \cap B) \to \mu (A)\mu(B)\ \mbox{as}\ i\to \infty . 
\]
The previous line is true, since $m_i \geq \epsilon_i \mathcal{H}_{n_i} > L_i$. 
\[
\mu (T^{m_i} A_{i,2} \cap B) \leq \mu (S_i^{\mathcal{H}_{n_i}-m_i} A_{i,2} \cap B) \leq 
\mu (S_i^{\mathcal{H}_{n_i}-m_i} A \cap B) \to \mu(A)\mu(B)\ \mbox{as}\ i\to \infty . 
\]
This is true, since $\mathcal{H}_{n_i}-m_i \geq \mathcal{H}_{n_i}-(\mathcal{H}_{n_i}-L_i)=L_i$. 
See Figure \ref{fig:verify-mix} for a pictorial describing this case. 

For the second case, let $D_{i,2}$ be the bottom $(m_i - \mathcal{H}_{n_i})$ levels 
of the left side of column $D_i$.  Let $D_{i,1}$ be the remaining levels or half-levels 
of $D_i$.  In this case, let $A_{i,j} = A\cap D_{i,j}$ for $j=1,2$.  
See Figure \ref{fig:verify-mix2} for a pictorial describing this case. 
Due to the uniform Cesaro property, $T^{m_i}A_{i,1}$ will mix with a set $B$.  
For $A_{i,2}$, we have 
\[
\mu (T^{m_i} A_{i,2} \cap B) \leq \mu (S_i^{m_i-\mathcal{H}_{n_i}} A_{i,2} \cap B) \leq 
\mu (S_i^{m_i-\mathcal{H}_{n_i}} A \cap B) \to \mu(A)\mu(B)\ \mbox{as}\ i\to \infty . 
\]
This previous line is true, since 
$m_i-\mathcal{H}_{n_i}\geq (\mathcal{H}_{n_i}+L_i) - \mathcal{H}_{n_i}=L_i$. 

Therefore, by applying Lemma \ref{GO} to the sequence $M_1$, we proved that 
$T$ is mixing on $M$. $\Box$

\section{$r$-thick sets and ${r}/{(r+1)}$-rigidity}
In this section, we generalize the results from the previous section 
to transformations with greater rigidity.  First, we introduce 
the notion of $r$-thick sets. 

\begin{defn}
Given a positive integer $r$, a set $K \subset \N$ is $r$-thick, 
if for each $\ell \in \N$, there exists $k \in \N$ such that 
\[
\{ jk + i : 1\leq j \leq r, -\ell \leq i \leq \ell \} \subset K . 
\]
\end{defn}
The definition of a $1$-thick set coincides with the usual 
definition of a thick set.  
Also, it is easy to verify that a $(n+1)$-thick set is necessarily 
$n$-thick.  However, for each $n \in \N$, there exists 
an $n$-thick set which is not $(n+1)$-thick. 

\begin{prop}
For any positive integer $r$, there exists an $r$-thick set 
$K \subset \N$ such that if $\alpha > \frac{r}{r+1}$ and 
a measure preserving transformation $T$ is $\alpha$-rigid, 
then $T$ is not mixing on the complement of $K$. 
\end{prop}

\begin{proof}
Let $K$ be any $r$-thick subset with the property that 
if $k \in K$, then $(r+1)k \notin K$.  For example, 
$K$ may be the set, 
\[
K = \bigcup_{n=1}^{\infty} [(r+1)^{(r+1)n}, (r+1)^{(r+1)n+1}-1] 
\cap \N . 
\]
Suppose invertible measure preserving transformation $T$ is 
$\alpha$-rigid with $\alpha > {r}/{(r+1)}$.  
Thus, there exists a sequence $r_n$ of positive integers such that 
for all measurable sets $A$, 
\[
\lim_{n\to \infty} \mu (T^{r_n}A\cap A) \geq \alpha \mu (A). 
\]
If $r_n \notin K$ for infinitely many $n$, then $T$ is not mixing 
on $K^c$ and we are done.  

Otherwise, $r_n \in K$ for sufficiently large $n$.  
Let $\beta$ be such that ${r}/{(r+1)} < \beta < \alpha$ and let 
$A$ be a set of positive measure such that $\mu (A) < (r+1)\beta - r$.  
Since $T$ is $\alpha$-rigid, there exists $N \in \N$ such that 
for $n \geq N$, $r_n \in K$ and 
\[
\mu (T^{r_n}A\cap A) > \beta \mu (A) . 
\]
Thus, for $n \geq N$, $\mu (T^{r_n}A\cap A^c) < (1-\beta) \mu (A)$.  
Hence, 
\begin{align}
\mu \big( T^{(r+1)r_n}A \cap A^c \big) \leq& 
\sum_{i=1}^{r+1} \mu \big( T^{ir_n}A \cap T^{(i-1)r_n} A^c \big) \\ 
<& \big( r+1 \big) \big( 1-\beta \big) \mu (A) \\ 
=& \Big( \big( r+1 \big) - \big( r+1 \big) \beta \Big) \mu(A) \\ 
=& \Big( 1 - \big( ( r+1) \beta - r \big) \Big) \mu(A) \\ 
<& \Big( 1 - \mu(A) \Big) \mu(A) . 
\end{align}
Therefore, $T$ is not mixing on $(r+1)r_n \in K^c$ for $n\geq N$. 
\end{proof} 

\begin{thm}
\label{v-r-staircase-thm}
If the complement of a set $M\subset \N$ contains an 
$r$-thick set for some $r\in \N$, then there exists an 
$\frac{r}{r+1}$-rigid transformation $T$ that is mixing on $M$. 
\end{thm}
\begin{proof}
Let $M \subset \N$ be such that its complement contains an $r$-thick subset.  
Thus, there exist increasing sequences $k_n$ and $\ell_n$ 
of natural numbers such that 
\[
\{ jk_n + i : 1\leq j\leq r, -\ell_n \leq i \leq \ell_n, n\in \N \} \subset M^c . 
\]
We define a transformation similar to the examples constructed 
in section \ref{stair-construct-section}.  
The only major change is that $r+1$ cuts are applied to the towers $D_i$ 
as defined in section \ref{mix-trans-rec-section}.  
A column $D_i$ of height $\mathcal{H}_{n_i}(\vec{r}+\vec{b},h_1)$ 
is defined in a similar manner, as well as $k=k_{y_i}$.  Add 
$k - \mathcal{H}_{n_i}(\vec{r}+\vec{b},h_1)$ spacers to the top of $D_i$, 
and rename this column $D_i$.  Cut this column into $r+1$ subcolumns 
of equal width and stack from left to right to obtain a single column 
of height 
\[
\big( r+1 \big) \big( k - \mathcal{H}_{n_i}(\vec{r}+\vec{b},h_1) \big) . 
\]
As in section \ref{mix-trans-rec-section}, 
return to the infinite recursion of alternating 
substages \ref{UC-substage} and \ref{NA-substage} with $\alpha=r_{n_i-1}$ 
to produce a mixing transformation denoted by $S_{i+1}$.  
The remainder of the proof follows 
in a manner similar to the proof of Theorem \ref{v-staircase-thm}.  
\end{proof}

The following theorem can be obtained as a straightforward modification 
of Theorems \ref{v-staircase-thm} and \ref{v-r-staircase-thm}.  

\begin{thm}
\label{rigid-cex-thm}
If the complement of a set $M\subset \N$ contains an 
$r$-thick set for each $r\in \N$, 
then there exists 
a rigid transformation $T$ that is mixing on $M$. 
\end{thm}

\begin{cor}
Given a sequence $M \subset \N$ with density zero. 
Then, there exists 
a rigid transformation $T$ that is mixing on $M$. 
\end{cor}
\begin{proof}
If a sequence $M$ has density zero
then for each $r \in \N$, 
the complement of $M$ contains an $r$-thick set.  Therefore, 
by Theorem \ref{rigid-cex-thm}, there exists a rigid transformation 
$T$ that is mixing on $M$. 
\end{proof}

\section{Modification of Friedman's infinite rank example}\label{Nat-C} 
In this section, we will present an example from the class of infinite rank transformations. The strategy is essentially based on the machinery of the so-called M-towers introduced by N. Friedman and Ornstein \cite{FO}. This machinery is used for example to produce a counter-example of $K$-systems  which is not Bernoulli \cite{O}, and partially mixing maps which are not mixing \cite{FO}. Nowadays, this method is known as the cutting and stacking machinery.  In the next subsection, we recall the basic definitions and tools from this method. We begin by stating our second main result.

\begin{thm}\label{Nat}Let $(s_i)$ be a sequence with density zero.
	Then there exist a infinite rank non-mixing map which is mixing on $\{s_i\}$.  
\end{thm}
Let us notice that  the complement of  any set of density zero contain  a shift of some set $\Big\{\sum \epsilon_jn_j,  \epsilon_j\in \{\pm 1,0\},  \sum |\epsilon_j|<\infty\Big\}$, where $(n_j)$ is a dissociated sequence. 
We recall that the sequence $(n_j)$ is a dissociated sequence if $n_j >2 \sum_{i=1}^{j-1}n_i$. An example of dissociated sequence is given by lacunary sequence, that is, the sequence for which we have $\frac{n_{j+1}}{n_j}>q\geq 3.$  As a consequence, we have
\begin{cor}
There exist a infinite rank non-mixing map which is mixing on $\{n^2, n \in \N\}.$
\end{cor}
\begin{proof}For any $n \geq 1$, we have $n^2 \equiv 1$ mod $3$. Therefore,
	$$\Big\{n^2, n \in \N\Big \} \cap  \Big\{\sum \epsilon_j5^{2j+1}, \epsilon_j\in \{\pm 1,0\},  \sum |\epsilon_j|<\infty\Big\}=\emptyset.$$
This finish the proof of the corollary.
\end{proof}
The proof of Theorem \ref{Nat} is based on the Friedman-Ornstein machinery. Precisely, our proof is essentially a modification of Friedman's example in his 1983 paper \cite{Fm}. Therein, among other results, he construct a transformation which is mixing on a sequence $(s_i)$ but not uniformly mixing on it. This later notion was introduced by N. Friedman in connection with 2-mixing and the famous  Rhoklin question which asked whether mixing implies mixing  of all order.  We notice that it is well known that the  Rhoklin problem can be reduced to the class of zero entropy. This will be discussed at the end of this section.\\

For the moment, let us state and prove a generalization of Ornstein's crucial mixing lemma.

\begin{lem}\label{GO}Let $(X,\mathcal{B},\mu, T)$ be a dynamical system and $(s_n)$ a sequence of positive integers. Suppose that 
	\begin{enumerate}
	\item \label{Ass1} For any irrational $\alpha$ in the circle, $\{s_i.\alpha\}$ is dense.
	\item  \label{Ass2} $T$ is totally ergodic (i.e., $T^n$ is ergodic for all $n \neq 0$) and ,
	\item  \label{Ass3} for any two measurable sets $A,B$
	\begin{align}\label{Orneq:1}
	\limsup \mu (T^{s_n}A\cap B) \leq K\mu(A)\mu(B),
	\end{align}  
	for some constant $K > 1$.
	\end{enumerate}
	Then $T$ is mixing along $(s_n)$.  
\end{lem}
\noindent Let us denote by  $\{p_n\}$ the sequence of primes, that is, $p_n$ is the $n$th prime and $\omega(n)$  the number of prime factors of $n$. We thus have as a direct consequence 
the following corollary.
\begin{cor}If $\{s_n\}$ is $\{n^k\}$, $k \geq 1$, or $\{p_n^\ell\}$, $\ell \geq 1$, or $\{\omega(n)\}$ and $T$ satisfy \eqref{Orneq:1}. Then $T$ is mixing along $(s_n)$.
\end{cor}
\begin{proof}By the well-know results, for any irrational number $\alpha$, the sequence $\{s_n \alpha\}$ is dense.
\end{proof}
Let us give the proof of Lemma \ref{GO}.
\begin{proof}[Proof of Lemma \ref{GO}]We first prove that $T$ is weakly mixing.  Note that
	since every power of $T$ is ergodic, $T$ can not admit
	eigenvalues which are roots of unity. Suppose $T$ admits
	eigenvalues of the form $e^{2\pi i a}$ for some irrational $a$ and $f$ its associated eigenfunction with $|f|=1$.
	Then, for any $c \in \C$, obviously  the set $\{x\;:\; f(x)=c\}$ has Lebesgue measure zero. Let $x \in [0,1)$ be such that for every arc $C$ containing $x$, $f^{-1}\big(C\big)$ has positive Lebesgue measure. Since $T$ is totally ergodic, the Lebesgue measure of $f^{-1}\big(C\big)$ 
	($m\big(f^{-1}\big(C\big)\big)$) can be made as small as we please by choosing $C$ small enough. Now, by the assumption \eqref{Ass1}, there is a subsequence $s_{n_j}$ such that $s_{n_j}.a \longrightarrow 1 \textrm{~~mod~} 1.$. Therefore,
	\begin{align*}
	m\Big(T^{-s_{n_k}}f^{-1}\big(C\big) \cap f^{-1}\big(C\big)\Big)&=
	m\Big(f^{-1}\big(e^{-2\pi i s_{n_k}.a}C\cap C\big)\Big)\\
	&  \tend{k} {+\infty}m\Big(f^{-1}\big(C\big)\Big)
		\end{align*}
	But, by our assumption \eqref{Ass2}, we have
	\begin{align*}
	\limsup m\Big(T^{-s_{n_k}}f^{-1}\big(C\big) \cap f^{-1}\big(C\big)\Big) \leq K \Bigg(m\Big(f^{-1}\big(C\big)\Big)\Bigg)^2. 
	\end{align*}
	If $m\Big(f^{-1}\big(C\big)\Big)$ is chosen small enough, we arrive to contradiction, proving that $T$ is weak mixing. To prove that $T$ is
	mixing along $(s_i)$. let us denote by $\Delta_n$ the graph joining given by $T^n$, by which we 
	mean the measure on $X\times X$ defined by
	$\Delta_n(A\times B) = \mu(T^nA\cap B)$, for any two sets
	$A, B \in {\mathcal{B}}$. Since the set of joining is compact in the weak 
	topology of joinings $ \displaystyle (i.e.~\rho_n \rightarrow \rho 
	~{\rm {iff~}} \rho_n(A \times B) \rightarrow \rho(A \times B),~~{\rm {for~each~rectangle~}}(A\times B))$, we can
	find a sequence of graph joining $(\Delta_{s_{n_j}})_{j=1}^\infty$
	converging to some joining, say,  $\lambda$. Whence, by our assumption \eqref{Ass2}, we get
	$$\lambda(A\times B) \leq K (\mu\otimes \mu) (A\times B).$$
	Let us define a joining $\psi$ by
	$$\mu \otimes \mu = \frac{1}{K}\lambda + (1 - \frac{1}{K}) \psi.$$
	But $\mu\otimes\mu$ is ergodic, so an extreme point in the space
	of invariant measures, whence $\lambda = \psi = \mu\otimes \mu$,
	so that for all $A, B \in \mathcal{B}$, $\mu(T^{s_n}A\cap B)
	\rightarrow \mu(A)\mu(B)$, and $T$ is mixing along $(s_n)$. This completes the
	proof of the Lemma.
\end{proof}

\subsection{ M-Towers by the cutting and independent staking method.}
Following \cite[section 2]{Fb}, an $n$-column is an ordered set of intervals $C=(I_1,\cdots,I_n)$, all of the same length. The width of $C$ is defined by $w(C)=|I_1|$, and the height of $C$ by $h(C)=n$. The base of $C$ is $b(C)=I_1$ and the top is $t(C)=I_n$. Associated with $C$ is the map $T_C$ which maps $I_{i}$ linearly to $I_{i+1}$, $1 \leq i \leq n-1$. So, $T_C$ is defined on $C\setminus I_n$. We thus have the following figure.
\vskip 1.0cm
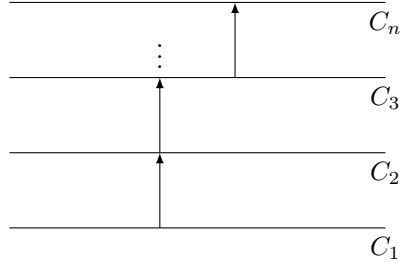
\begin{center}
\begin{tikzpicture}
\draw (0,0) -- (5,0)node[below]{$C_1$};
\draw (0,1) -- (5,1) node[below]{$C_2$} ;
\draw[->,>=latex] (2,0)--(2,1);
\draw (0,2) -- (5,2)  node[below]{$C_3$};
\draw[->,>=latex] (2,1)--(2,2) node[above]{$\vdots$};
 \draw (0,3) -- (5,3)node[below]{$C_n$};
  \draw[->,>=latex] (3,2)--(3,3) ;
\end{tikzpicture}
\end{center}
A tower $G$ is an ordered set of disjoint columns $C_1,C_2, \cdots C_l$. The top of $G$ is
$t(G)=\bigcap_{j=1}^{l}t(C_j)$ and its base is  $b(G)=\bigcap_{j=1}^{l}b(C_j)$.  The width of $G$ is
 $w(G)=|t(G)|=|b(G)|$. The transformation $T_G$ consists of $T_{C_j},$ $1 \leq j \leq l$. A level in column in $C_j$ is considered as a level in $G$. $G$ can be seen also as union of all levels in it. Thus, $T_G$ is defined on $G\setminus t(G).$

\section{Proof of Theorem \ref{Nat}} \label{Proof-II} We will gives a construction  based on Friedman's construction. 
It will be carried out by induction.\\

We start by fixing an $M$-towers $T_{G}$ as  support of mixing time. Therefore, for any $\epsilon>0$, there is a positive integer $N(G,\epsilon)$ such that for any $n >N(G,\epsilon)$, we have
\begin{align}\label{M1}
	\Big|\mu(T_{G}^{n^2}I \cap J)-\frac{\mu(I) \mu(J)}{\mu(G)}\Big|<\epsilon,
\end{align}
where $I,J$ are levels in $G.$  Whence, by construction, we can choose $k=k(G,\epsilon,N)$ so large that $T$ extends $T_k$. We thus have
\begin{align}\label{M2}
\Big|\mu(T_{G}^{n^2}I \cap J)-\frac{\mu(I) \mu(J)}{\mu(G)}\Big|<\epsilon,\; \; N(G,\epsilon) \leq n^2 \leq N
\end{align}
by \eqref{M1}. Let $G$ be a tower such that $w(C) \in \mathbb{Q},$ for each $C$ in $G$. We can assume without loss of generality that the denominator of all $w(C)$ is the same. We cut each $C$ in $G$ into sub-columns all of the same width $w$, and we stacked consecutively  (in the usual fashion) to form one column of with $w$ that we denote by $w(G)$. Notice that if $I$ is a level in $G$, then $I=\cup_{\alpha \in R} I_\alpha^*,$ where $R$ is finite and $I_\alpha$ are in $C(G)$. Let $C$ be a column and $r$ a positive integer. We cut $C$ into $r$ sub-columns of equal with $w(C)/r$, and we stack in the usual fashion to form a single column denoted by $S_r(C)$ with height $rh(C)$ and width $\frac{w(C)}{r}.$  Now, for $\epsilon>0$ and $t \in \N$, we choose $r \geq \frac{\epsilon}{t}$ and let $T$ be any extension of $T_{S_r(C)}$. It follows that if $J$ is a level in $C$, then, by construction, we have
\begin{align}\label{Ergo}
	\mu\Big(\bigcap_{j=0}^{t}T^{jh(C)}J\Big) \geq (1-\epsilon)\mu(J).
\end{align}      
Now, in our $M$-tower at the stage $n$, we have $G_n$ such that $w(G_n)\setdef w_n \in \Q$. Moreover, each level $I$ in $G_i$, $1 \leq i \leq n$, satisfy 
$$I=\cup_{I^* \subset G_n}I^*.$$
Let $L_n$ be the total number of levels in $G_n$ and $\epsilon_n>0$ such that  
$\epsilon_n<\frac{w_n}{100 L_n^2}$,  we choose also $N_n$ and $k_n$ such that \eqref{M1} and \eqref{M2} holds. Let $r_n \geq \max\{k_n,N_n/\epsilon_n\}$ and construct $G_{n,1}=S_{r_n}(G_n)$. Now, since $r_n \geq k_n$, if $T$ extend $T_{G_{n,1}}$ we have
\begin{align}\label{Ext-1}
\Big|\mu(T^{i^2}I \cap J)-\frac{\mu(I) \mu(J)}{\mu(G_n)}\Big|<\epsilon_n, \; \; N_{n-1} \leq i^2 \leq N_{n}.
\end{align} 
Now, by constructing $G_{n,2}=C(G_{n,1})$, we see that for each Borel $A$ such that 
$A=\bigcup_{I \subset G_n}I$ we have $A=\bigcup_{I \subset G_{n,2}, I \subset A}I$. Moreover, by the choice of $r_n$, we get that $T^{i^2}A=\bigcup_{I} I \cup A_{\epsilon_{n},i}, $
where $\mu(A_{\epsilon_{n},i})<\epsilon_{n}$. This is due to the fact that

$$T^{N_n}A=\bigcup_{J \subset G_{n,2}} J \cup A_{\epsilon_{n},N_n}. $$  
Now, let $(t_j)$ be a sequence of positive number such that for each $n \geq 1$, we have 
$$\frac{b(n-1)}{t_n}<\epsilon_n,$$
where $b(n)=\sum_{i=1}^{n}t_i.$
Choose $r_n \geq \frac{t_n}{\epsilon_{n}}$, and construct $G_{n,3}=S_{r_n}(G_{n,2})$. Put $h(G_{n,2})=h_n$. Therefore, if $J$ is a level in $G_{n,2}$ and $T$ extend $T_{G_{n,3}}$, then, by \eqref{Ergo}, we have
\begin{align}\label{Ergo2}
	\mu\Big(\bigcap_{j=0}^{t_n}T^{jh_n}J\Big) \geq (1-\epsilon_n)\mu(J).
\end{align}
Put $s_{j,i}=i^2+\big(j-b(n-1)\big)h_n,\;  b(n-1) \leq j <b(n),\; N_{n-1}\leq i^2 \leq N_n.$  Then, 
\begin{align}\label{Mix-1}
\mu\Big(\bigcap_{j=0}^{t_n}T^{s_{j,i}}A\Big) \geq (1-\epsilon_n)\mu(A).
\end{align}
We finish the construction, by cutting $G_{n,3}$ into two equal columns and adding an extra interval above one column. We thus get the tower $G_{n+1}$ is $M$-towers consisting of two columns with heights differing by one.  By construction each level $I$ in $G_{n}$ satisfy $I=\bigcup_{J \in G_{n+1}}J,$ and $w(G_{n+1}) \in \Q$. We further assume that $(\frac{n}{\sqrt{N_n}})$ converge to $1$. We need now to prove that the mixing holds along the square and the map is not mixing. For mixing along square, observe that for any Borel sets $A,B$ union of level of some $G_i$, $1 \leq i\leq m$, $m$ is a fixed integer, we have that $A,B$ are union of levels in $G_n$. Therefore, for any $i \in [\sqrt{N_{n-1}},\sqrt{N_n}]$, we have
\begin{align}
&	\big|\mu(T^{i^2})A\cap B-\mu(A)\mu(B)\big| \\
&\leq \big|\mu(T^{i^2}A\cap B)-\frac{\mu(A)\mu(B)}{\mu(G_n)}\big|+
\Big|\frac{\mu(A)\mu(B)}{\mu(G_n)}-\mu(A)\mu(B)\Big|\\
&\leq L_n^2 \epsilon_{n}+\Big|\frac{1}{\mu(G_n)}-1\Big|
\end{align}
But $$L_n^2 \epsilon_{n} \tend{n}{+\infty}0,$$ and 
$$\Big|\frac{1}{\mu(G_n)}-1\Big| \tend{n}{+\infty}0.$$
We thus get that $T$ is mixing along $\{i^2, i \in N\},$ since the Borel algebra is generated by the levels in $G_k$, $k \geq 1$. To establish that $T$ is not mixing,
we apply  
 \eqref{Mix-1} for instance to $\Big((s_{j,\lceil N_n\rceil})_{1 \leq j \leq b(n-1)}\Big)_{n \in \N}$ to see that
\begin{align}\label{Mix-3}
	\mu(T^{(j-i)h_n}A \cap A) \geq (1-2\epsilon_n) \mu(A).
\end{align}
Whence $T$ is not mixing. 

~~



 


\section{Examples from probabilistic origin}\label{Gaussian}
In this section we will gives a class of examples from probabilistic origin which satisfy Theorem \ref{Nat}.
The classical dynamical systems of probabilistic origin are given by the so-called Gaussian transformation. It is defined as a shift map $S$ on $\Omega=\mathbb{R}^{\N}$ equipped with the Gauss measure $\mu$, that is, $(S(\omega(s))_{s\in \R}=(\omega(s+1)_{s\in \R})$. The projection on $s$-th coordinate is denoted by $X(s),$ so $X((\omega(s))_{s\in \R})=\omega_s$.  We recall that $\mu$ is said to be a Gauss measure if
the joint distribution of any family of variables $(X_{s_1},\cdots,X_{s_r})$ is an
$r$-dimensional Gauss distribution, that is, for any Borel sets $B_1,\cdots,B_r$, we have
\begin{align}\label{Gauss-1}
	\mu(X(s_1)\in B_1,\cdots, X(s_r)\in B_r)=
	\frac{1}{(2\pi)^{\frac{r}{2}}.|C|}\int_{B_1\times \cdots \times B_r } \exp\big(-\frac{1}{2}.{}^{t}{x}.C^{-1}.x\big) dx,
\end{align}
where $x=(x_1,x_2,\cdots,x_r)$, $dx=dx_1\times \cdots dx_r$, $C$ is the matrix of covariance given by
$$C(s_i,s_j)=\E\big(X_{s_i}X_{s_j}\big)=\int X_{s_i}(\omega)X_{s_j}(\omega)d\mu(\omega)),\; \; i,j=1,\cdots,r.$$
It is known that the Gaussian distribution  is well defined by its matrix of covariance and its mean vectors $(m(s_1),\cdots,m(s_r)$ given by
  $$m(s)=\int X_s(\omega) d\mu(\omega).$$
But, since $\mu$ is invariant with respect to $S$, we have
$$m(s)=m(0) \; \; \textrm{and} \; \; C(s_1,s_2)= C(s_i+t,s_j+t), \;\; \forall t \in \N.$$
Whence, for any $s_1,s_2 \in \N$ 
$$C(s_1,s_2)=C(0,s_2-s_1).$$ 
Therefore, by Herglotz-Bochner-Khinchin theorem, there exist a finite measure on the circle $S^1$ such that 
$$C(s+t,s)=\widehat{\sigma}(t)=\int z^t d\sigma(z).$$
The measure $\sigma$ is known as the spectral measure of the Gauss measure $\mu$, and it is completely determines the Gauss measure $\mu$ (under the assumption that $m=0$). The study of ergodicity, weak-mixing, mixing, rigidity are reduced to the study of the property of $\sigma$. For a nice account on the spectral analysis of Gaussian transformation, we refer to \cite[Chap. 14]{CFS}.

Let us further notice that for a given $\sigma$, we generate a Gauss transformations with the desired properties. Hence, we are reduced to produce a probability measure on the circle for which the Fourier coefficients along a set of density zero converge to zero and  $\sigma$ is not in Rajchman class. We recall that $\sigma$ is in Rajchman class if its Fourier coefficients converge to zero. \\

Our strategy here is reduced to produce some tools from harmonic analysis which will allows us to produce such measures. For that, we recall the notion of Rajchman sets and its characterization.
\begin{defn}A Radon measure $\mu$ on the unit circle is said to be a Rajchman measure if 
$$\widehat{\mu}(n)\tend{|n|}{+\infty}0.$$	
Let $R$ be a subset of $\Z$. $R$ is said to be a Rajchman set if for any Radon measure $\mu$ on the unit circle, if $\widehat{\mu}(n) \tend{n, n \in \Z \setminus R}{+\infty}0$ then 
	$\widehat{\mu}(n)\tend{|n|}{+\infty}0,$ that is $\mu$ is Rajchman measure.
\end{defn} 
The characterization of Rajchman sets is due to B. Host and F. Parreau and its based on the notion of classical Riesz product which we recall briefly in the following paragraph.
\paragraph{\textbf{Classical Riesz products.}} The Riesz products are some kind of probability measures defined on circle. The original example was introduced by F. Riesz in his 1918's paper \cite{FR} to produce a continuous probability measure whose Fourier coefficients does not vanish at infinity.  They are powerful source of counterexamples and their definition are intimately related to the notion of the so-called dissociation. This later notion is defined as follows.\\

Consider the following two products:
\begin{eqnarray*}
	(1+z)(1+z) &=& 1+z+z +z^2 = 1 +2z +z^2,\\
	(1+z)(1+z^2) &=& 1+z+z^2+z^3.
\end{eqnarray*}
In the first case we group terms with the same power of $z$, while in the second case all the powers of $z$ in the formal expansion are distinct. In the second case we say that the polynomials $1 + z$ and $1 + z^2$ are dissociated. More generally we say that a set of trigonometric polynomials is dissociated if in the formal expansion of product of any finitely many of them, the powers of $z$ in the non-zero terms are all distinct \cite{Abd-Nad}.\\
Let $(n_j)_{j \geq 0}$ be a sequence of positives integers such that for all $j \geq 1,$ 
$$n_j>2\sum_{k=0}^{j-1}n_k,$$
and $(c_j)_{j \geq 0}$ a sequence of complex numbers such that $|c_j| \leq \frac12,$ for all $j$. Put
$$P_j(z)=1+2\Rea\big({c_j z^{n_j}}\big), z \in S^1, j=1,2,\cdots$$
Then it is well known that $(P_j))$ are dissociated. We further have 
$$\prod_{j=0}^{k}P_j(z)=\sum_{(\epsilon_l)_{l=0}^{k} \in\{-1,0,1\}^{k+1} }
\big(\prod_{m=0}^{k}\mathcal{C}^{\epsilon_m}(c_m)\big)z^{\sum_{j=0}^{k}\epsilon_j n_j},$$
Where 

$$\mathcal{C}^{\epsilon}(z)=\begin{cases}
z & \textrm{if}\; \epsilon=1,\\
1 & \textrm{if}\; \epsilon=0,\\
\overline{z} & \textrm{if}\; \epsilon=-1.
\end{cases}	
$$
We thus get 
$$\prod_{j=0}^{k}P_j(z)=\sum_{j=-N_k}^{N_k}d_jz^j, N_k=\sum_{j=0}^{k}n_j.,$$
with 
$$d_j=\begin{cases}
\prod_{m=0}^{k}\mathcal{C}^{\epsilon_m}(c_m) & \textrm{if}\; j=\sum_{i=0}^{k}\epsilon_i n_i ,\\
0 & \textrm{if~not.}
\end{cases}	
$$
For each $k \geq 0$, put $$\mu_k=\prod_{j=0}^{k}P_j(z)dz.$$
Then $\mu_k$ are all a probability measures, and it is well known that $(\mu_k)$ converge to $\mu$ a probability measure on the circle called nowadays classical Riesz product. Indeed, we have
$$\widehat{\mu}(j)=\begin{cases}
\prod_{m \in J}\mathcal{C}^{\epsilon_m}(c_m) & \textrm{if}\; j=\sum_{i \in J}\epsilon_i n_i ,
J \subset \N \textrm{~~finite.} \\
0 & \textrm{if~not.}
\end{cases}	
$$
We thus have that the support of $\widehat{\mu}$ is
$$\Big\{ \sum_{i \in J}\epsilon_i n_i ,
J \subset \mathcal{P}_f(\N),\epsilon_i \in \{-1,0,1\}, i \in J, \Big \}, $$
where $\mathcal{P}_f(\N)$ is the set of all finite set of $\N$. For details on the Riesz products and its generalization, we refer to \cite{Abd-Nad} and the references therein. In \cite{Abd-Nad}, it is noticed that the spectral type of the rank one maps give raise to a generalized Riesz products and the classical Riesz product products can be used to produced an infinite rank one, that is, a rank one map acting on infinite measure space. Therefore, as a consequence of the proof of Theorem \ref{v-staircase-thm}, we have

\begin{cor} Let $M \subset \Z$ be such that its complement is thick. Then, there is a continuous generalized Riesz product $\sigma$ such that $\sigma$ is not a Rajchman measure and 
\begin{align}
		\widehat{\sigma}(m)\tend{m \in M}{+\infty}0.
	\end{align}
\end{cor}

We will now present a characterization of Rajchman sets based on the classical Riesz product. This is due to Host-Parreau \cite[Théorème 3]{HP1}, \cite[Corollary 2]{HP2}.

\begin{thm} A subset $R \subset \Z$ is a Rajchman set if and only if for any $t \in \Z$, for any dissociated sequence $(n_j)$ we have
\begin{align}\label{HP}
\Big(\Big\{\sum_{j \in \N}\epsilon_j n_j, \epsilon_j \in \{\pm 1,0\}, \sum_{j}|\epsilon_j|<+\infty \Big\}+t\Big)
\cap R^c \neq \emptyset.
\end{align} 
\end{thm}
In other words, $R$ it does not contain a shift of the Fourier support of a classical Riesz product.
\noindent Let us emphasize that Host and Parreau stated their arithmetical characterization for the so-called set of type $M_0$.  The set $C$ is a set of type $M_0$ if for any finite measure $\mu$ such $\textrm{supp}(\widehat{\mu}) \subset C$ , we have $$\lim_{|n|\to \infty} |\widehat{\mu}(n)|=0.$$ 
As a consequence, they obtain a  characterization for the set continuity. We recall that  set $C$ is said to be a set of continuity if for each number $\epsilon>0$ there is a $\delta > 0$ such that, if $\mu$ is finite measure with $\|\mu\| \leq 1$, then the following condition holds
$$\limsup_{n \in C^c} |\widehat{\mu}(n)|<\delta \Longrightarrow \limsup_{n \in C} |\widehat{\mu}(n)|<\epsilon.$$
Precisely, Host and Parreau established that $C \subset \Z$ is a set of continuity if and only if 
there is $N \in \N$ such that for any $t \in \Z$, we have
\begin{align}\label{HPn}
	\Big(\Big\{\sum_{j \in \N}\epsilon_j n_j, \epsilon_j \in \{\pm 1,0\}, \sum_{j}|\epsilon_j| \leq N \Big\}+t\Big)
	\cap C^c \neq \emptyset.
\end{align}
A set $E$ which does not satisfy \eqref{HPn} are is said to staisfy \emph{Rajchman dissociated property}. For a full exposition on Host-Parreau arithmetical characterization of Rajchaman and continuity sets, we refer to \cite[Chap. VI]{HP3}.\\

\noindent{}We are now able to stat the following theorem.
\begin{thm}\label{R}Let $E \subset \Z$ be a set which satisfy the Rajchman dissociated property. Then, there exist a non-mixing Gaussian transformation that is mixing on the complement of $E$.
\end{thm}
\begin{proof}Let $E \subset \Z$ be a set which satisfy the Rajchman dissociated property. Then, $E$ is not a continuity set. Therefore, there is  $\sigma$ a probability measure on the circle such that 
	$$\widehat{\sigma}(n) \tend{n, n \in \Z \setminus E}{+\infty}0.$$  
Take a Gaussian process $(X_n)_{n \in \Z}$ such that 
$$\E(X_0 X_n)=\widehat{\sigma}(n) .$$
Then, the associated Gaussian maps satisfied the desired properties.
\end{proof}
As a corollary, we have 
\begin{cor} There exists a non-mixing Gaussian transformation that is mixing on the square.
\end{cor}
At this point, let us emphasize that the class of continuity sets contain the Rajchman sets, the Sidon sets, and the Riesz sets.

 
However, We warn the reader that the notion of Sidon set in harmonic is completely different than that in number theory. 
\begin{defn}A  subset of integers  $S$ is a Sidon set if for every $f \in C(\T)$ such that 
	$\widehat{f}(n)=0$, for all $n \in S^c$, we have $\sum_{n\in\Z}|\widehat{f}(n)|<+\infty.$  
\end{defn}
It is easy to see that by applying closed-graph theorem that $S$ is a Sidon set if and only if, there exists a constant $K>0$ such that 
$$\sum_{n\in\Z}|\widehat{f}(n)| \leq K \|f\|_{\infty},$$
for every $f \in C(\T)$ with  $\textrm{supp}(\widehat{f}) \subset S.$ According to Host-Parreau characterization, Sidon sets, Hadamard sets,  Riesz sets, weak Rajchman sets are all Rajchman sets.
\begin{rem} All the examples produced here are in the class of zero entropy. Let us notice that Bergleson and al. addressed the famous question of Rhoklin whether mixing implies mixing of all order.  We would like to point out that by Rhoklin-Sinai observation, any system is relatively $K$-system whith respect to its Pinsker factor, and it is well known that mixing implies mixing of all order for $K$-system. Therefore, Rhoklin question is reduced to the class of zero entropy. Now, by Ryzikhov's theorem \cite{R-mixing}, all mixing transformations with finite rank are mixing of all order. This reduced the question to the mixing transformation with infinite rank. 
\end{rem}

\section{Mixing group action on thin subset without mixing property}\label{Groupe}
In this section, we will present a generalization of the previous result to the countable abelian group action $G$ equipped with its discrete topology. We denote its dual group by $\widehat{G}$ which is a compact abelian group. Let $(X, \mathcal{B}, \mu)$ be a Borel probability space. For each $g \in G$, let $T_g: X \rightarrow X$ be a measure-preserving transformation. The family $(T_g)_{g \in G}$ can be seen as the group action of $G$ on $X$. To avoid a heavy notation, we will write  the map $T_g$  simply $g$ when no confusion can arise. The associated Koopman is given by $U_g(f)=f \circ T_g,$ $f \in L^2(X, \mathcal{B}, \mu)$. In our setting, $G$ is a topological group and the map $g\in G\longmapsto T_g$ is assumed to be continuous with respect to weak operator topology, that is, for any $f,h \in L^2(X,\mu)$, the map $g \in G \mapsto \langle U_gf,h \rangle$ is continuous map. Therefore, the dynamical system $(X,\mathcal{B}, \mu,G)$ give a unitary representation of $G$ in $L^2(\mathcal{B}, \mu)$.    We recall that the action of $G$ is mixing on a subset $M$ of $G$, if for any sequence
$\{g_n\} \subset M$, for any Borel set $A,B$, we have
$$\lim_{n\to \infty} \mu(g_nA \cap B)=\mu(A)\mu(B).$$
It is weak mixing if the only measurable solutions to the following equation 
$$f \circ T_g = \chi(g)f ~~~\textrm{a.e.},$$
where $\chi \in  \widehat{G}$  are the constant functions. \\

\noindent We start by stating the group version of Theorem \ref{v-staircase-thm}.
\begin{thm}\label{vs-group-action-1}
Let $M$ be a subset of $G$. Then, there is a non-mixing map on it which is mixing on $M$ if and only if its complement is a thick set.
\end{thm}
We recall that the subset $R \subset G$ is called a thick set, if for every compact set $K \subset G$, we have $\ds \bigcap_{k\in K}k^{-1}R \neq \emptyset$.  A subset $S$ is syndetic if there is a compact set 
$K \subset G$ with $G=KS.$ For the correspond notions and a nice account, we refer to \cite{BHM} and \cite{HS}. 
Let us notice also that Theorem \ref{vs-group-action-1} can be obtained as a consequence of Host-Parreau characterization of a set of continuity. Indeed, their theorem was stated as follows
\begin{lem}[\cite{HP2}]\label{HPg} A set $E$ is a set of continuity if and only if for some positive integer $n$, $E$ does not contain the translate of a set of words on length $n$, that is,  $ \alpha \Big\{\ds\prod_{j=1}^{+\infty}\theta^{\epsilon_j}, \epsilon_j \in \{\pm 1,0\}, \sum_{j=1}^{+\infty}|\epsilon_j| \leq n\Big\}$, for any $\alpha \in G$ and infinite dissociate set $\{\theta_j\}$.   
\end{lem}
We recall  that the subset $\{\theta_j\}$ is dissociate if for each integer $N$, 
$$  \prod_{j=1}^{N}\theta^{\epsilon_j}=e \textrm{~~with~~} \epsilon_j=0,\pm 1, \pm 2
\Longrightarrow \theta^{\epsilon_j}=1, \forall j=1,\cdots,N.$$
We have also the following 
\begin{thm}\label{vs-mixing-group} Let $M \subset G$ be a subset such that its complement is a thick set. Then, there is a non-mixing group action $(T_g)_{g \in G}$ which is mixing on $M$. 
\end{thm} 
To produce an example of group action on $M$, we will follows the same ideas as in section \ref{Gaussian}, thus, we will produce a Gaussian action of $G$. Let $\Omega=\R^G.$ The coordinate maps $X_g(\omega)=\omega_g$ are said to Gaussian, if for any finite set $F \subset G$,
$(X_g)_{g \in G}$ is a Gaussian vectors. The action is stationary if, for any $g \in G$, the map $g \mapsto X_{gh}$ does not change the combined expectation of $X_g$.  As, for the $\Z$-action, the action is determined by the 
covariance, that is, $\E(X_g \overline{X_h})$ which form a positive define matrix. We assume also that the mean is zero. Therefore, the action is completely determined by 
$\E(X_e \overline{X_g})$ which must be positive define function on $G$. But, by Bochner's theorem, any positive define function of a locally compact group abelian group  generate a positive measure $\sigma$ on $\widehat{G}$ such that 
$$\E(X_e \overline{X_g})=\widehat{\sigma}(g).$$
Conversely, every  a positive measure $\sigma$ on $\widehat{G}$ correspond to a positive define function on $G$ and hence define a stationary Gaussian process. It is well known that the spectrum of the dynamical system $(\Omega, \mu, (T_g)_{g \in G})$ can be expressed in term of the measure $\sigma$ and It\^{o}'s Chaos as for $\Z$-actions ( see \cite[p.356-373]{CFS} \cite[p. 67-68]{Ki}). Now, we present the proof of Theorem \ref{vs-mixing-group}.
\begin{proof}[\textbf{Proof of Theorem \ref{vs-mixing-group}}.] Since the complement of $M$ is thick, it is follows by Lemma \ref{HPg} that it is not a set of continuity, that is, there is a probability measure which is a Riesz product such that the support of its Fourier transform is $M^c$. Now, by taking the corresponding Gaussian group action, we obtain a non-mixing action which is mixing on $M$. 
\end{proof} 
\begin{rem}The similar construction to that of section \ref{vs-staircase-alpha} can obtained by using the Ornstein random construction. It can be also extended to the case of countable group action by applying the $C-F$ algorithm introduced by A. del Junco in \cite{Ad} (this generalized the notion of funny rank one to the action of group). The $C-F$ mixing action à l'Ornstein are constructed in \cite{Ad} and \cite{M}.  
\end{rem}

Let us notice that our methods as some limitation, we thus ask:
\begin{question}
	It is possible to extend our results to the action of amenable countable groups or, by applying the same methods, to the locally compact abelian groups.
\end{question}

The extension to the locally compact abelian groups seems more accessible since the extension of the notions of Riesz products, Sidon Sets, Riesz sets, Rachjman and continuity sets are available (see for instance \cite{KS}, \cite{MH}, \cite[Appendix A, Section A.1]{GH} and the references therein).

\section*{\textbf{Acknowledgement.}}
	The first author wishes to express his thanks to Jean-Paul Thouvenot and Mahesh Nerurkar for a discussion on the subject. He would like to express also his thanks to the organizers of the workshop ``Operator theoretic aspects of ergodic theory" in live Wuppertal \& per zoom (17-19 February 2022) where a part of this work was announced. He wishes also to extend  his special thanks to Rainer Nagel. The authors would like to thank Valery Ryzhikov and Vitaly Bergelson for their valuable comments and suggestions.      
\appendix 

\section{Case where $m_i \in M_1$}

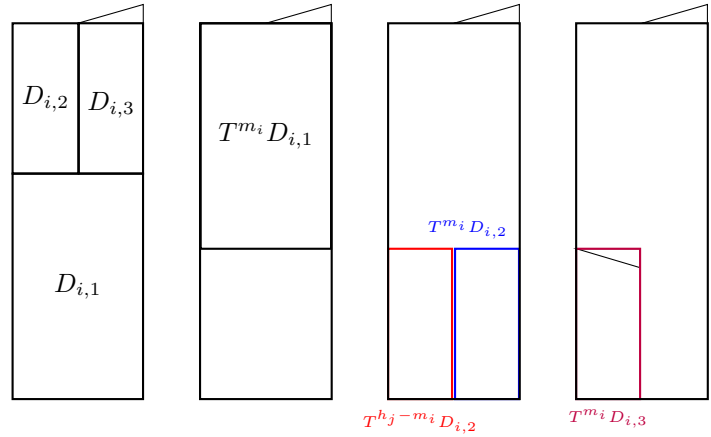
\begin{figure}[!htb]
\begin{center}
\begin{tikzpicture}

\draw (1.25,5) node[anchor=north]{}
  -- (2.12,5) node[anchor=north]{}
  -- (2.12,5.25) node[anchor=south]{}
  -- cycle;

\node (rect) at (.82,4)[draw, thick, minimum width=0.5cm, minimum height=2cm]{ 
\begin{minipage}{0.64cm}
\parbox{0.64cm}{\begin{center}$D_{i,2}$\end{center}}
\end{minipage}
}; 

\node (rect) at (1.69,4)[draw, thick, minimum width=0.5cm, minimum height=2cm]{ 
\begin{minipage}{0.62cm}
\parbox{0.62cm}{\begin{center}$D_{i,3}$\end{center}}
\end{minipage}
}; 

\node (rect) at (1.25,1.5)[draw, thick, minimum width=1.5cm, minimum height=3cm]{ 
\begin{minipage}{1.5cm}
\parbox{1.5cm}{\begin{center}$D_{i,1}$\end{center}}
\end{minipage}
};

\draw (3.75,5) node[anchor=north]{}
  -- (4.62,5) node[anchor=north]{}
  -- (4.62,5.25) node[anchor=south]{}
  -- cycle;

\node (rect) at (3.75,2.5)[draw, thick, minimum width=1.75cm, minimum height=5cm]{};

\node (rect) at (3.75,3.5)[draw, thick, minimum width=1.5cm, minimum height=3cm]{ 
\begin{minipage}{1.5cm}
\parbox{1.5cm}{\begin{center}$T^{m_i}D_{i,1}$\end{center}}
\end{minipage}
}; 

\draw (6.25,5) node[anchor=north]{}
  -- (7.12,5) node[anchor=north]{}
  -- (7.12,5.25) node[anchor=south]{}
  -- cycle;

\node (rect) at (5.80,1)[draw=red, thick, minimum width=0.85cm, minimum height=2cm, label=below:\textcolor{red}{\tiny $T^{h_j-m_i}D_{i,2}$}]{}; 

\node (rect) at (6.69,1)[draw=blue, thick, minimum width=0.85cm, minimum height=2cm, label=above:\textcolor{blue}{\tiny $T^{m_i}D_{i,2}\qquad$}]{}; 

\node (rect) at (6.25,2.5)[draw, thick, minimum width=1.75cm, minimum height=5cm]{}; 

\draw (8.75,5) node[anchor=north]{}
  -- (9.62,5) node[anchor=north]{}
  -- (9.62,5.25) node[anchor=south]{}
  -- cycle;

\draw (7.88,2) node[anchor=north]{}
  -- (8.72,2) node[anchor=north]{}
  -- (8.72,1.75) node[anchor=south]{}
  -- cycle;

\node (rect) at (8.30,1)[draw=purple, thick, minimum width=0.85cm, minimum height=2cm, label=below:\textcolor{purple}{\tiny $T^{m_i}D_{i,3}$}]{}; 

\node (rect) at (8.75,2.5)[draw, thick, minimum width=1.75cm, minimum height=5cm]{}; 

\end{tikzpicture}
\end{center}
\caption{Half-Rigid Staircase: $m_i \leq \mathcal{H}_{n_i}-L_i$}
\label{fig:verify-mix}
\end{figure}

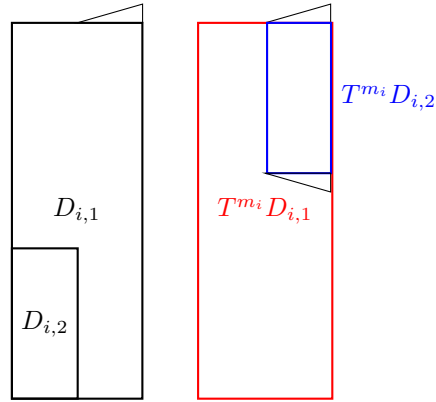
\begin{figure}[!htb]
\begin{center}
\begin{tikzpicture}

\draw (1.25,5) node[anchor=north]{}
  -- (2.12,5) node[anchor=north]{}
  -- (2.12,5.25) node[anchor=south]{}
  -- cycle;

\node (rect) at (.82,1)[draw, thick, minimum width=0.5cm, minimum height=2cm]{ 
\begin{minipage}{0.64cm}
\parbox{0.64cm}{\begin{center}$D_{i,2}$\end{center}}
\end{minipage}
}; 


\node (rect) at (1.25,2.5)[draw, thick, minimum width=1.5cm, minimum height=5cm]{ 
\begin{minipage}{1.5cm}
\parbox{1.5cm}{\begin{center}$D_{i,1}$\end{center}}
\end{minipage}
};

\draw (3.75,5) node[anchor=north]{}
  -- (4.62,5) node[anchor=north]{}
  -- (4.62,5.25) node[anchor=south]{}
  -- cycle;

\node (rect) at (3.75,2.5)[draw=red, thick, minimum width=1.75cm, minimum height=5cm]{
\begin{minipage}{1.55cm}
\parbox{1.55cm}{\begin{center}\textcolor{red}{$T^{m_i}D_{i,1}$}\end{center}}
\end{minipage}
};

\node (rect) at (4.20,4.0)[draw=blue, thick, minimum width=0.85cm, minimum height=2cm, label=right:\textcolor{blue}{$T^{m_i}D_{i,2}$}]{};

 \draw (3.75,3) node[anchor=north]{}
   -- (4.62,3) node[anchor=north]{}
   -- (4.62,2.75) node[anchor=south]{}
   -- cycle;








\end{tikzpicture}
\end{center}
\caption{Half-Rigid Staircase: $m_i \geq \mathcal{H}_{n_i}+L_i$}
\label{fig:verify-mix2}
\end{figure}

\end{document}